\documentclass[11pt,english,a4paper]{smfart}
\usepackage[english]{babel}
\usepackage{amssymb,xspace}
\usepackage{amstext}
\usepackage{smfthm}
\theoremstyle{plain}
\usepackage{amsbsy,amssymb,amsfonts,latexsym}

\usepackage{enumerate}
\usepackage{graphics}

\date{\today}

\title[IP-Dirichlet measures via generalized Riesz products]{IP-Dirichlet measures and IP-rigid dynamical systems: an approach via generalized Riesz products}

\author{Sophie Grivaux}
\address{CNRS,
Laboratoire Paul Painlev\' e, UMR 8524, Universit\'e  Lille 1, Cit\' e Scientifique, 59655 Villeneuve d'Ascq
Cedex, France}
\email{grivaux@math.univ-lille1.fr}



\def\T{\ensuremath{\mathbb T}}
\def\R{\ensuremath{\mathbb R}}
\def\Z{\ensuremath{\mathbb Z}}

\def\C{\ensuremath{\mathbb C}}

\def\N{\ensuremath{\mathbb N}}

\newcommand{\op}{operator}

\newcommand{\wrt}{with respect to}

\newcommand{\proba}{probability}

\newcommand{\mea}{measure}

\newcommand{\mpt}{measure-preserving transformation}

\newcommand{\wmx}{weakly mixing}

\newcommand{\rg}{rigid}

\newcommand{\ds}{dynamical system}

\newcommand{\ipd}{IP-Dirichlet}
\newcommand{\seq}{sequence}

\newcommand{\ifff}{if and only if}

\newtheorem{theorem}{Theorem}[section]

\newtheorem{lemma}[theorem]{Lemma}

\newtheorem{proposition}[theorem]{Proposition}

\newtheorem{corollary}[theorem]{Corollary}

{\theoremstyle{definition}}

{\theoremstyle{definition}}

{\theoremstyle{definition}\newtheorem{definition}[theorem]{Definition}}

{\theoremstyle{definition}}

{\theoremstyle{definition}}

{\theoremstyle{definition}\newtheorem*{FFC Criterion}{Frequent
Faber-hypercyclicity Criterion}}
\newtheorem*{Hypercyclicity Criterion}{Hypercyclicity Criterion}
{\theoremstyle{definition}\newtheorem*{GS Criterion}{Godefroy-Shapiro
Criterion}}
\def\piednote#1{\let\oldfn=\thefootnote\def\thefootnote{}\footnote{\noindent#1}%
\addtocounter{footnote}{-1}\def\thefootnote{\oldfn}}

\def\wh{\hat}

\begin{document}

\begin{abstract}
If $(n_{k})_{k\ge 1}$ is a strictly increasing \seq\ of integers, a continuous probability measure $\sigma  $ on the unit circle $\T$ is said to be IP-Dirichlet with respect to $(n_{k})_{k\ge 1}$ if $\hat{\sigma  }(\sum_{k\in F}n_{k})\to 1 $ as $F$ runs over all non-empty finite subsets $F$ of $\N$ and the minimum of $F$ tends to infinity. IP-Dirichlet measures and their connections with IP-rigid dynamical systems have been investigated recently by Aaronson, Hosseini and Lema\'nczyk. We simplify and generalize some of their results, using an approach involving generalized Riesz products.
\end{abstract}
\maketitle

\section{Introduction}

We will be interested in this paper in \ipd\ \proba\ \mea s on the unit circle $\T=\{\lambda \in \C \textrm{ ; } |\lambda |=1\}$ \wrt\ a strictly increasing \seq\ $(n_{k})_{k\ge 1}$ of positive integers. Recall that a \proba\ \mea\ $\mu $ on $\T$ is said to be a \emph{Dirichlet \mea}\ when there exists a strictly increasing \seq\ $(p_{k})_{k\ge 1}$ of integers such that the monomials $z^{p_{k}}$ tend to $1$ on $\T$ as $k$ tends to infinity \wrt\  the norm of $L^{p}(\mu)$, where $1\le p<+\infty$. This is equivalent to requiring that the Fourier coefficients $\hat{\mu}(p_{k})$ of the \mea\ $\mu$ tend to $1$ as $k$ tends to infinity. If $(n_{k})_{k\ge 1}$ is a (fixed) striclty increasing \seq\ of integers, we say that $\mu $ \emph{is a Dirichlet \mea\ \wrt\ the \seq}\ $(n_{k})_{k\ge 1}$ if  $\hat{\mu}(n_{k})\to 1$ as $k \to +\infty$. Let $\mathcal{F}$ denote the set of all non-empty finite subsets of $\N$. The \mea\ $\mu$ is said to be \emph{\ipd\ \wrt\ the \seq}\ $(n_{k})_{k\ge 1}$ if $$\hat{\mu}(\sum_{k\in F}n_{k})\to 1 \quad\ \textrm{ as } \min(F)\to +\infty, \; F\in \mathcal{F}.$$
In other words: for all $\varepsilon >0$ there exists a $k_{0}\ge 0$ such that whenever $F$ is a finite subset of $\{k_{0}, k_{0}+1, \ldots\}$,
$$|\hat{\mu}(\sum_{k\in F}n_{k})-1|\ge \varepsilon .$$ 
\par\smallskip
Our starting point for this paper is the work \cite{AHL} by Aaronson, Hosseini and Lema\'nczyk, where \ipd\ \mea s are studied in connection with rigidity phenomena for \ds s. Let $(X, \mathcal{B},m)$ denote a standard non-atomic \proba\ space and let $T$ be a \mpt\ of $(X,\mathcal{B},m)$. Let again $(n_{k})_{k\ge 1}$ be a strictly increasing sequence of integers.

\begin{definition}\label{def1}
 The transformation $T$ is said to be \emph{\rg\ \wrt}\ $(n_{k})_{k\ge 1}$ if
  $m(T^{-n_{k}}A\triangle A)\to 0$ as $n_{k}\to +\infty $ for all sets $A\in \mathcal{B}$, or, equivalently, if for all functions $f\in L^{2}(X, \mathcal{B},m)$,
  $||f\circ T^{n_{k}}-f||_{L^{2}(X, \mathcal{B},m)}\to 0$ as ${k}\to +\infty $.
\end{definition}

If we denote by $\sigma  _{T}$ the restricted spectral type of $T$, i.e. the spectral type of the Koopman \op\ $U_{T}$ of $T$ acting on functions of $L^{2}(X, \mathcal{B},m)$ of mean zero, then it is not difficult to see that $T$ is \rg\ \wrt\ $(n_{k})_{k\ge 1}$ \ifff\ $\sigma  _{T}$ is a Dirichlet \mea\ \wrt\ the sequence $(n_{k})_{k\ge 1}$.

\par\smallskip

Rigidity phenomena for \wmx\ transformations have been investigated recently in the papers \cite{BDLR} and \cite{EG}, where in particular the following question was considered: given a sequence $(n_{k})_{k\ge 1}$ of integers, when is it true that there exists a \wmx\ transformation $T$ of some \proba\ space $(X, \mathcal{B},m)$ which is \rg\ \wrt\ $(n_{k})_{k\ge 1}$? When this is true, we say that $(n_{k})_{k\ge 1}$ is a \emph{\rg ity \seq}. It was proved in \cite{BDLR} and \cite{EG} that $(n_{k})_{k\ge 1}$ is a \rg ity \seq\ \ifff\ there exists a continuous \proba\ \mea\ $\sigma  $ on $\T$ which is Dirichlet \wrt\ $(n_{k})_{k\ge 1}$.

\par\smallskip

It is then natural to consider IP-rigidity for (\wmx)\ \ds s. This study was initiated in \cite{BDLR} and continued in \cite{AHL}.

\begin{definition}\label{def2}
The system $(X, \mathcal{B},m;T)$ is said to be \emph{IP-\rg\ \wrt\ the sequence} $(n_{k})_{k\ge 1}$ if for every $A\in \mathcal{B}$,
 $$m(T^{\sum_{k\in F}n_{k}}A\triangle A)\to 0 \quad \textrm{ as } \min (F)\to+\infty, \; F\in \mathcal{F}.$$
\end{definition}

Just as with the notion of \rg ity, $T$ is IP-\rg\ \wrt\ $(n_{k})_{k\ge 1}$ \ifff\ $\sigma  _{T}$ is an \ipd\ \mea\ \wrt\ $(n_{k})_{k\ge 1}$. Moreover, if we say that $(n_{k})_{k\ge 1}$ is an \emph{IP-\rg ity \seq}\ when there exists a \wmx\ \ds\ $(X, \mathcal{B},m;T)$ which is IP-\rg\ \wrt\ $(n_{k})_{k\ge 1}$, then IP-\rg ity \seq s can be characterized in a similar fashion as \rg ity sequences (\cite[Prop. 1.2]{AHL}): $(n_{k})_{k\ge 1}$ is an IP-\rg ity \seq\ \ifff\ there exists a continuous \proba\ \mea\ $\sigma  $ on $\T$ which is IP-Dirichlet \wrt\ $(n_{k})_{k\ge 1}$.

\par\smallskip

IP-Dirichlet \mea s are studied in detail in the paper \cite{AHL}, and one of the important features which is highlighted there is the connection between the existence of a \mea\ which is \ipd\ \wrt\ a certain \seq\ $(n_{k})_{k\ge 1}$ of integers, and the properties of the subgroups $G_{p}((n_{k}))$ of the unit circle associated to $(n_{k})_{k\ge 1}$: for $1\le p<+\infty$,
$$G_{p}((n_{k}))=\{\lambda \in \T \textrm{ ; } \sum_{k\ge 1}|\lambda ^{n_{k}}-1|^{p}<+\infty\}$$
and for $p=+\infty$
$$G_{\infty}((n_{k}))=\{\lambda \in\T \textrm{ ; } |\lambda ^{n_{k}}-1|\to 0 \textrm{ as } k\to +\infty\}.$$

The main result of \cite{AHL} runs as follows:

\begin{theorem}\cite[Th. 2]{AHL}\label{th0}
 Let $(n_{k})_{k\ge 1}$ be a strictly increasing sequence of integers. If $\mu $ is a \proba\ \mea\  on $\T$ which is \ipd\ \wrt\ $(n_{k})_{k\ge 1}$, then $\mu (G_{2}((n_{k})))=1.$
\end{theorem}

The converse of Theorem \ref{th0} is false \cite[Ex. 4.2]{AHL}, as one can construct a sequence $(n_{k})_{k\ge 1}$ and a \proba\ \mea\ $\mu $ on $\T$ which is continuous, supported on $G_{2}((n_{k}))$ (which is uncountable), and not \ipd\ \wrt\ $(n_{k})_{k\ge 1}$. On the other hand, if $\mu $ is a continuous \proba\ \mea\  such that $\mu(G_{1}((n_{k})))=1$, then $\mu$ is \ipd\ \wrt\ $(n_{k})_{k\ge 1}$ \cite[Prop. 1]{AHL}. Again, this is not a necessary and sufficient condition for being \ipd\ \wrt\ $(n_{k})_{k\ge 1}$ \cite{AHL}: if 
$(n_{k})_{k\ge 1}$ is the \seq\ of integers defined by $n_{1}=1$ and $n_{k+1}=kn_{k}+1$ for each $k\ge 1$, then there exists a continuous probability \mea\ $\sigma  $ on $\T$ which is \ipd\ \wrt\ $(n_{k})_{k\ge 1}$, although $G_{1}((n_{k}))=\{1\}$. Numerous examples of \seq s $(n_{k})_{k\ge 1}$ \wrt\ which there exist \ipd\ continuous \proba\ \mea s are given in \cite{AHL} as well. For instance, such \seq s are characterized among \seq s $(n_{k})_{k\ge 1}$ such that $n_{k}$ divides $n_{k+1}$ for each $k$, and among \seq s which are denominators of the best rational approximants $\frac{p_{k}}{q_{k}}$ of an irrational number $\alpha \in (0,1)$, obtained via the continued fraction expansion. It is also proved in \cite{AHL} that sequences $(n_{k})_{k\ge 1}$ such that the series $\sum_{k\ge 1}(n_{k}/n_{k+1})^{2}$ is convergent admit a continuous \ipd\ \proba\ \mea.

\par\smallskip

Our aim in this paper is to simplify and generalize some of the results and examples of \cite{AHL}. We first present an alternative proof of Theorem \ref{th0} above, which is completely elementary and much simpler than the proof of \cite{AHL} which involves Mackey ranges over the dyadic adding machine. 
We then present a rather general way to construct \ipd\ \mea s via generalized Riesz products. The argument which we use is inspired by results from \cite{Pa} and \cite[Section 4.2]{HMP}, where generalized Riesz products concentrated on some $H_{2}$-subgroups of the unit circle are constructed. Proposition \ref{Prop3} gives a bound from below on the Fourier coefficients of these Riesz products, and this enables us to obtain in Proposition \ref{Prop5} a sufficient condition on sets $\{n_{k}\}$ of the form
\begin{equation}\label{eq2}
 \{n_{k}\}=\bigcup_{k\ge 1}\{p_{k},q_{1\, k}\, p_{k},\dots,q_{r_{k},\, k}\, p_{k}\},
\end{equation}
where the $q_{j,k}$, $j=1,\ldots r_{k}$, are positive integers and the \seq\ $(p_{k})_{k\ge 1}$ is such that  $p_{k+1}>q_{r_{k},k}p_{k}$ for each $k\ge 1$, for the existence of an associated continuous generalized Riesz product which is \ipd\ \wrt\ $(n_{k})_{k\ge 1}$. This condition is best possible (Proposition \ref{Prop7}). As a consequence of Proposition \ref{Prop5}, we retrieve and improve a result of \cite{AHL} which runs as follows: if $(n_{k})_{k\ge 1}$ is such that there exists an infinite subset $S$ of $\N$ such that 
\[
\sum_{k\in S}\dfrac{n_{k}}{n_{k+1}}<+\infty \quad\textrm{and}\ n_{k}\vert n_{k+1}\ \textrm{for each}\ k\not\in S,
\]
then there exists a continuous probability measure $\sigma $ on $\T$ which is \mbox{IP-Dirichlet} with respect to $(n_{k})_{k\ge 1}$. This result is proved in \cite{AHL} by constructing a rank-one weakly mixing system which is \mbox{IP-rigid}  with respect to $(n_{k})_{k\ge 1}$. Here we get a ``dynamical system-free'' proof of this statement, where the condition $\sum_{k\in S}(n_{k}/n_{k+1})<+\infty$ is replaced by the weaker condition $\sum_{k\in S}(n_{k}/n_{k+1})^{2}<+\infty$.

\begin{theorem}\label{Cor6}
 Let $(n_{k})_{k\ge 1}$ be a strictly increasing sequence of integers for which there exists an infinite subset $S$ of $\N$ such that
\[
\sum_{k\ge 1}\, \, \Bigl( \dfrac{n_{k}}{n_{k+1}}\Bigr)^{2}<+\infty \quad\textrm{and}\ n_{k}\vert n_{k+1}\ \textrm{for each}\ k\not\in S.
\]
Then there exists a continuous generalized Riesz product $\sigma $ on $\T$ which is \mbox{IP-Dirichlet} with respect to $(n_{k})_{k\ge 1}$.
\end{theorem}

Using again sets of the form (\ref{eq2}), we
 then show that the converse of Theorem \ref{th0} is false in the strongest possible sense, thus strengthening Example 4.2 of \cite{AHL}:

\begin{theorem}\label{th1}
There exists a strictly increasing sequence $(n_{k})_{k\ge 1}$ of integers such that $G_{2}((n_{k}))$ is uncountable, but no continuous \proba\ \mea\ is \ipd\ \wrt\ $(n_{k})_{k\ge 1}$.
\end{theorem}

The last section of the paper gathers some observations concerning the Erd\"os-Taylor \seq\ $(n_{k})_{k\ge 1}$ defined by $n_{1}=1$ and $n_{k+1}=kn_{k}+1$, which is of interest in this context.
\par\smallskip
\textbf{Notation:}
In the whole paper, we will denote by $\{x\}$ the distance of the real number 
$x$ to the nearest integer, by $\lfloor x \rceil$ the integer which is closest to $x$ (if there are two such integers, we take the smallest one), and by $\langle x \rangle$ the quantity
$x-\lfloor x \rceil$. Lastly, we denote by $[x]$ the integer part of $x$.

\section{An alternative proof of Theorem \ref{th0}}

 Let $(n_{k})_{k \ge 1}$ be a strictly increasing sequence of integers.
 Suppose that the measure $\mu $ on $\T$ is \mbox{IP-Dirichlet} with respect to $(n_{k})_{k\ge 1}$. For every
$\varepsilon >0$ there exists an integer $k_{0}$ such that for all sets $F\in \mathcal{F}$ with $\min(F)\ge k_{0}$, $\left|\widehat{\mu}\left(\sum_{k\in F}n_{k}\right)-1\right|\le \varepsilon $. For every integer $N\ge k_{0}$, consider the quantities 
\[
\prod_{k=k_{0}}^{N}\frac{1}{2}\left(1+\lambda ^{n_{k}}\right)={2^{-(N-k_{0}+1)}}\sum_{F\subseteq\{k_{0},\dots,N\}}\lambda ^{\sum_{k\in F}n_{k}}.
\]
The notation on the righthand side of this display means that the sum is taken over all (possibly empty) finite subsets $F$ of $\{k_{0},\dots,N\}$.
Integrating with respect to $\mu $ yields that 
\[
\int_{\T }\prod_{k=k_{0}}^{N}\frac{1}{2}\left(1+\lambda ^{n_{k}}\right)d\mu (\lambda )={2^{-(N-k_{0}+1)}}\sum_{F\subseteq\{k_{0},\dots,N\}}\widehat{\mu}\bigl(\sum_{k\in F}n_{k}\bigr),
\]
so that
\begin{equation}\label{eq1}
 \Bigl|\int_{\T }\prod_{k=k_{0}}^{N}\frac{1}{2}\left(1+\lambda ^{n_{k}}\right)d\mu (\lambda )-1\Bigr|\le {2^{-(N-k_{0}+1)}}\sum_{F\subseteq\{k_{0},\dots,N\}}\Bigl|\widehat{\mu}\bigl(\sum_{k\in F}n_{k}\bigr)-1\Bigr|\le \varepsilon .
\end{equation}
Let now $$C=\{\lambda \in\T\ \textrm{ ; the infinite product}\ \prod_{k=1}^{+\infty}\frac{1}{2}|1+\lambda ^{n_{k}}|\ \textrm{converges to a non-zero limit}\}.$$ Observe that the set $C$ does not depend on $\varepsilon $ nor on $k_{0}$. For every $\lambda \in \T \setminus C$, the quantity
$\prod_{k=k_{0}}^{N}\frac{1}{2}|1+\lambda ^{n_{k}}|$ tends to $0$ as ${N\to +\infty }$, and so by the dominated convergence theorem we get that
\[
\int_{\T \setminus C}\prod_{k=k_{0}}^{N}\dfrac{1}{2}\left(1+\lambda ^{n_{k}}\right)d\mu (\lambda )\rightarrow 0 \quad \textrm{ as } {N\to+\infty }.
\]
It then follows from (\ref{eq1}) that 
\[
\Bigl|\limsup_{N\to+\infty }\int_{C}\prod_{k=k_{0}}^{N}\dfrac{1}{2}\left(1+\lambda ^{n_{k}}\right)d\mu (\lambda )-1\Bigr|\le\varepsilon 
\]
so that
\[
\limsup_{N\to+\infty }\Bigl|\int_{C}\prod_{k=k_{0}}^{N}\dfrac{1}{2}\left(1+\lambda ^{n_{k}}\right)d\mu (\lambda )\Bigr|\ge 1-\varepsilon .
\]
But $$\Bigl|\int_{C}\prod_{k=k_{0}}^{N}\frac{1}{2}(1+\lambda ^{n_{k}})\, d\mu (\lambda )\Bigr|\le \mu (C),$$ 
hence $\mu (C)\ge 1-\varepsilon $. This being true for any choice of $\varepsilon $ in $(0,1)$, $\mu (C)=1$, and so the product $\prod_{k\ge 1}\frac{1}{2}|1+\lambda ^{n_{k}}|$ converges to a non-zero limit almost everywhere with respect to the measure $\mu $.
If we now write elements $\lambda \in C$ as $\lambda =e^{2i\pi\theta }$, $\theta \in[0,1 )$, we have 
$$\prod_{k\ge  1}\frac{1}{2}|1+\lambda ^{n_{k}}|=\prod_{k\ge  1}|\cos(\pi\theta {n_{k}})|.$$ Since $0<|\cos(\pi\theta {n_{k}})|\le 1   $ for all $k\ge 1$, this means that the series $\sum_{k\ge 1}1-|\cos(\pi\theta n_{k})|$ is convergent.  In particular $\{\theta n_{k}\}\to 0$ as $k\to +\infty$. As the quantities $1-|\cos(\pi\theta n_{k})|$ and $\frac{\pi^{2}}{2}\{{\theta }n_{k}\}^{2}$ are equivalent as $k\to +\infty$, we obtain that the series  $\sum_{k\ge 1}\{{\theta }n_{k}\}^{2}$ is convergent.
But
\[
\bigl|1-\lambda ^{n_{k}}\bigr|^{2}=\bigl|1-e^{2i\pi\theta {n_{k}}}\bigr|^{2}\le 4\pi^{2} \{\theta {n_{k}}\}^{2},
\]
and it follows from this that the series $\sum_{k\ge 1}\bigl|1-\lambda ^{n_{k}}\bigr|^{2}$ is convergent as soon as $\lambda $ belongs to $C$. This proves our claim.

\section{IP-Dirichlet generalized Riesz products}

Our aim is now to give conditions on the \seq\ $(n_{k})_{k\ge 1}$ which imply the existence of a generalized Riesz product which is continuous and \ipd\ \wrt\ $(n_{k})_{k\ge 1}$. For information about classical and generalized Riesz products, we refer for instance the reader to the papers \cite{Pa} and \cite{HMP} and to the books \cite{HMP2} and \cite{Queff}.

\begin{proposition}\label{Prop3}
Let $(n_{k})_{k\ge 1}$ be a strictly increasing sequence of integers. Suppose that there exists a strictly increasing sequence $(m_{k})_{k\ge 1}$ of integers such that
\begin{equation}\label{Eq1}
 n_{k+1}-2\sum_{j=1}^{k}m_{j}n_{j}\ge 1\quad\textrm{for each}\ k\ge 1,
\end{equation}
and
\begin{equation}\label{Eq2}
 n_{k+1}-2\sum_{j=1}^{k}m_{j}n_{j}\longrightarrow+\infty \quad\textrm{as}\ k\longrightarrow+\infty. 
\end{equation}
For each $k\ge 1$, let $p_{k}\ge 1$ be an integer such that ${p_{k}\pi}\le{m_{k}+2}$. There exists a continuous generalized Riesz product $\sigma $ on $\T$ such that for every finite subset $F\in \mathcal{F}$ and every integers $j_{k}$ in $\{1,\dots,p_{k}\}$, $k\in F$, one has
\begin{align}
 \hat{\sigma }\Bigl(\sum_{k\in F}j_{k}n_{k}\bigr)&\ge\prod_{k\in F}\Bigl( 1-3\pi^{2}\Bigl( \dfrac{p_{k}}{m_{k}+2}\Bigr)^{2}\Bigr)\label{Eq3}
\intertext{and}
\hat{\sigma }\Bigl(\sum_{k\in F}n_{k}\bigr)&=\prod_{k\in F}\cos\Bigl( \dfrac{\pi}{m_{k}+2}\Bigr).\label{Eq4}
\end{align}
\end{proposition}
\begin{proof}
 For any integer $k\ge 1$, consider the polynomial $P_{k}$ defined by 
\[
P_{k}(e^{2i\pi t})=\dfrac{2}{m_{k}+2}\Bigl|\sum_{j=1}^{m_{k}+1}\sin\Bigl( \dfrac{j\pi }{m_{k}+2}\Bigr)e^{2i\pi jt}\Bigr|^{2},\quad t\in[0,1].
\]
Each $P_{k}$ is a nonnegative trigonometric polynomial. Its spectrum is the set 
$\{-m_{k},\dots, m_{k}\}$ and a straightforward computation shows that $\wh{P_{k}}(0)=1$. Condition (\ref{Eq1}), which is a dissociation condition, implies that the probability measures 
$\prod_{k=1}^{N} P_{k}\bigl( e^{2i\pi n_{k}t}\bigr)d\lambda (t)$ (where $\lambda $ denotes here the normalized Lebesgue \mea\ on $\T$) converge in the $w^{*}$ topology as $N\to+\infty$ to a probability measure $\sigma $ on $\T$, and that for each  $F\in \mathcal{F}$  and each integers $j_{k}\in\{-m_{k},\dots,m_{k}\}$, $k\in F$,
\[
\wh{\sigma }\Bigl( \sum_{k\in F}j_{k}n_{k}\Bigr)=\prod_{k\in F}\wh{P_{k}}(j_{k}),
\]
while $\wh{\sigma }(n)=0$ when $n$ is not of this form. In particular
\[
\wh{\sigma }\Bigl( \sum_{k\in F}n_{k}\Bigr)=\prod_{k\in F}\wh{P_{k}}(1).
\]
Before getting into precise computation of these Fourier coefficients, let us prove that $\sigma $ is a continuous measure: this follows from condition (\ref{Eq2}). If
\[
\sum_{j=1}^{k}m_{j}n_{j}<n<n_{k+1}-\sum_{j=1}^{k}m_{j}n_{j},
\]
then $\wh\sigma (n)=0$. So the Fourier transform of $\sigma $ vanishes  on successive intervals $I_{k}$ of length $l_{k}=n_{k+1}-2\sum_{j=1}^{k}m_{j}n_{j}-1$. Since $l_{k}$ tends to infinity with $k$ by (\ref{Eq2}), it follows from the Wiener theorem that $\sigma $ is continuous.
\par\medskip 
Let us now go back to the computation of the Fourier coefficients $\wh{\sigma }\Bigl( \sum_{k\in F}j_{k}n_{k}\Bigr)$. For each $p\in\{1,\dots,m_{k}\}$, we have
\begin{equation}\label{ee1}
\wh{P_{k}}(p)=\dfrac{2}{m_{k}+2}\sum_{j=1}^{m_{k}+1-p}\sin\Bigl( \dfrac{(j+p)\pi }{m_{k}+2}\Bigr)\sin\Bigl( \dfrac{j\pi }{m_{k}+2}\Bigr).
\end{equation}
Standard computations yield the following expression for $\wh{P_{k}}(p)$:
\begin{eqnarray}\label{ee2}
\wh{P_{k}}(p)&=& \dfrac{1}{m_{k}+2}
 \Bigl((m_{k}+2-p)\cos\Bigl(\dfrac{p\pi }{m_{k}+2}\Bigr)+\sin\Bigl(\dfrac{p\pi }{m_{k}+2}\Bigr)\,\cdot\,\dfrac{\cos\bigl(\frac{\pi }{m_{k}+2}\bigr)}{\sin\bigl(\frac{\pi }{m_{k}+2}\bigr)}\Bigr)\\
&=&\dfrac{1}{m_{k}+2}\Bigl((m_{k}+2-p)\cos\Bigl( \dfrac{p\pi }{m_{k}+2}\Bigr)+\cos\Bigl( \dfrac{(p-1)\pi }{m_{k}+2}\Bigr)\,\cdot\,\cos\Bigl( \dfrac{\pi }{m_{k}+2}\Bigr)\notag\\
&&+\sin\Bigl( \dfrac{(p-1)\pi }{m_{k}+2}\bigr)\,\cdot\,\dfrac{\cos^{2}\bigl( \frac{\pi }{m_{k}+2}\bigr)}{\sin\bigl( \frac{\pi }{m_{k}+2}\bigr)}\notag\\
&=&\dots =\dfrac{1}{m_{k}+2}\Bigl( (m_{k}+2-p)\cos\Bigl( \dfrac{p\pi }{m_{k}+2}\Bigr)\notag\\
&&\qquad+\sum_{j=1}^{p}\cos\Bigl( \dfrac{(p-j)\pi }{m_{k}+2}\Bigr)\cos^{j}\Bigl( \dfrac{\pi }{m_{k}+2}\Bigr)\Bigr)\notag.
\end{eqnarray}
Observe now  that for every $x\in[0,1]$, $\cos x\ge 1-x^{2}\ge 0$. For each $k\ge 1$, let $p_{k}\ge 1$ be an integer such that $p_{k}\pi\le m_{k}+2$, and let $p$ belong to the set $\{1,\dots,p_{k}\}$. Then $(p-j)\pi \le m_{k}+2$ for every $j\in\{0,\dots,p-1\}$. Thus
\[
\cos\Bigl(\dfrac{p\pi }{m_{k}+2}\Bigr)\ge 1-\pi ^{2}\dfrac{p^{2}}{(m_{k}+2)^{2}}\quad\textrm{and}\quad
\cos\Bigl(\dfrac{(p-j)\pi }{m_{k}+2}\Bigr)\ge 1-\pi ^{2}\dfrac{(p-j)^{2}}{(m_{k}+2)^{2}}\cdot 
\]
Moreover, $\cos^{j}x\ge (1-x^{2})^{j}\ge 1-jx^{2}$ for every $x\in[0,1]$ and every $j\ge 1$, so that 
\[
\cos^{j}\Bigl( \dfrac{\pi }{m_{k}+2}\Bigr)\ge 1-\pi ^{2}\dfrac{j}{(m_{k}+2)^{2}}\cdot
\]
Putting things together, we obtain the estimate
\begin{align*}
\wh{P_{k}}(p)&\ge\dfrac{1}{m_{k}+2}\Bigl((m_{k}+2-p)\Bigl(1-\pi ^{2}\dfrac{p^{2}}{(m_{k}+2)^{2}}\Bigr)\\
&\hspace*{3cm}+\sum_{j=1}^{p}\Bigl(1-\pi ^{2}\dfrac{(p-j)^{2}}{(m_{k}+2)^{2}}\Bigr)
\Bigl(1-\pi ^{2}\dfrac{j}{(m_{k}+2)^{2}}\Bigr)\Bigr)\cdot
\end{align*}
Now, for every $j\in\{1,\dots,p-1\}$,
\begin{align*}
 \Bigl(1-\pi^{2}\,\dfrac{(p-j)^{2}}{(m_{k}+2)^{2}}\Bigr)
\Bigl(1-\pi^{2}\dfrac{j}{(m_{k}+2)^{2}}\Bigr)&=1-\pi^{2}\dfrac{(p-j)^{2}+j}{(m_{k}+2)^{2}}+\pi^{4}\,\dfrac{j(p-j)^{2}}{(m_{k}+2)^{2}}\\
&\ge 1-\pi^{2}\dfrac{(p-j)^{2}+j}{(m_{k}+2)^{2}}\ge 1-2\pi^{2}\dfrac{p^{2}}{(m_{k}+2)^{2}}\cdot
\end{align*}
Summing over $j$ and putting together terms, we eventually obtain that
\begin{align*}
 \wh{P_{k}}(p)&\ge \dfrac{1}{m_{k}+2}\Bigl( ({m_{k}+2}-p)\Bigl(1-\pi^{2}  \dfrac{p^{2}}{(m_{k}+2)^{2}}\Bigr)+p-2\pi^{2} \dfrac{p^{3}}{(m_{k}+2)^{2}} \Bigr)\\
 &\ge 1-\dfrac{1}{m_{k}+2}({m_{k}+2}-p)\pi^{2}\Bigl( \dfrac{p}{m_{k}+2}\Bigr)^{2}-2\pi^{2} \Bigl( \dfrac{p}{m_{k}+2}\Bigr)^{3},
\end{align*}
i.e. that
\begin{align*}
 \wh{P_{k}}(p)&\ge 1-\pi^{2}\Bigl( \dfrac{p}{m_{k}+2}\Bigr)^{2}-2\pi^{2}\Bigl(  \dfrac{p}{m_{k}+2}\Bigr)^{3}\\
&\ge 1-\pi^{2}\Bigl( \dfrac{p_{k}}{m_{k}+2}\Bigr)^{2}-2\pi^{2}\Bigl(  \dfrac{p_{k}}{m_{k}+2}\Bigr)^{3}
\quad \textrm{ for each } p\in\{1,\ldots, p_{k}\}\\
&\ge 1-3\pi^{2}\Bigl( \dfrac{p_{k}}{m_{k}+2}\Bigr)^{2}.
\end{align*}
Assertion (\ref{Eq3}) follows directly from the fact that $\wh{\sigma }\bigl( \sum_{k\in F}j_{k}n_{k}\bigr)=\prod_{k\in F}\wh{P_{k}}(j_{k})$. Assertion (\ref{Eq4}) is straightforward:
the expression in the first line of the display (\ref{ee2}) applied to $p=1$ yields that $\wh{P_{k}}(1)=\cos(\pi /(m_{k}+2))$. This finishes the proof of Proposition \ref{Prop3}.
\end{proof}

Proposition \ref{Prop3} may appear a bit technical at first sight, but it turns out to be quite easy to apply. As a first example, we use it to obtain another proof of a result of \cite[Prop. 3.2]{AHL}:

\begin{corollary}\label{Cor4}
 Let $(n_{k})_{k\ge 1}$ be a strictly increasing sequence of integers such that the series 
$\sum_{k\ge 1} (n_{k}/n_{k+1})^{2}$ is convergent. There exists a continuous generalized Riesz product $\sigma $ on $\T$ which is \mbox{IP-Dirichlet} with respect to $(n_{k})_{k\ge 1}$. 
\end{corollary}

\begin{proof}
 Without loss of generality we can assume that $\sum_{k\ge 1} (n_{k}/n_{k+1})^{2}<1/36$. Let 
$(\varepsilon _{k})_{k\ge 1}$ be a sequence of real numbers with $0<\varepsilon _{k}<1/2$ for each $k\ge 2$, with $\varepsilon _{1}=0$, going to zero as $k$ tends to infinity, and such that 
\[
\sum_{k\ge 1}\,\Bigl( \dfrac{1}{\varepsilon _{k+1}}\dfrac{n_{k}}{n_{k+1}}\Bigr)^{2}<\dfrac{1}{9}\cdot
\]
Then $\varepsilon _{k+1} n_{k+1}/n_{k}>3>2+\varepsilon _{k}$, so that if we define $m_{k}=[(\varepsilon _{k+1}n_{k+1}-\varepsilon _{k}n_{k})/2n_{k}]$ for each $k\ge 1$, each $m_{k}$ is a positive integer. Moreover
\[n_{k+1}-2\sum_{j=1}^{k}m_{j}n_{j}\ge n_{k+1}-(\varepsilon _{k+1}n_{k+1}-\varepsilon _{1}n_{1})=(1-\varepsilon _{k+1})n_{k+1}
\]
which tends to infinity as $k$ tends to infinity, and is always greater than $1$ because $\varepsilon _{k+1}<1/2$ and $n_{k+1}\ge 2$ for each $k\ge 1$. Proposition \ref{Prop3} applies with this choice of the sequence $(m_{k})_{k\ge 1}$ and yields a continuous generalized Riesz product $\sigma $ which satisfies
\[
\wh{\sigma }\Bigl( \sum_{k\in F}n_{k}\Bigr)=\prod_{k\in F}\cos\Bigl(\dfrac{\pi }{m_{k}+2} \Bigr)\quad\textrm{for each } F \in \mathcal{F}.
\]
Now $m_{k}$ is equivalent as $k$ tends to infinity to the quantity $\varepsilon _{k+1} n_{k+1}/2n_{k}$, so that the series $\sum_{k\ge 1}1/(m_{k}+2)^{2}$ is convergent. Hence the infinite product $\prod_{k\ge 1}\cos\bigl( \pi /(m_{k}+2)\bigr)$ is convergent. For any $\varepsilon >0$, let $k_{0}$ be such that 
$\prod_{k\ge k_{0}}\cos\bigl( \pi /(m_{k}+2)\bigr)\ge 1-\varepsilon $. If $F\in \mathcal{F}$ is such that $\min (F)\ge k_{0}$,
\[
\wh{\sigma }\Bigl( \sum_{k\in F}n_{k}\Bigr)=\prod_{k\in F}\cos\Bigl( \dfrac{\pi }{m_{k}+2}\Bigr)\ge\prod_{k\ge k_{0}}\cos\Bigl( \dfrac{\pi }{m_{k}+2}\Bigr)\ge 1-\varepsilon,
\]
and this proves that $\sigma $ is \mbox{IP-Dirichlet} with respect to $(n_{k})_{k\ge 1}$. 
\end{proof}

\section{An application to a special class of sets $\{n_{k}\}$}

Proposition \ref{Prop3} applies especially well to a particular class of sequences $(n_{k})_{k\ge 1}$, which we now proceed to investigate.

\begin{proposition}\label{Prop5}
 Let $(p_{l})_{l\ge 1}$ be a strictly increasing sequence of integers. For each $l\ge 1$, let $(q_{j,l})_{j=0,\dots,r_{l}}$ be a strictly increasing finite sequence of integers with $q_{0,l}=1$, and set 
$q_{l}=q_{0,l}+q_{1,l}+\cdots+q_{r_{l}, l}$. Suppose that $p_{l+1}>q_{r_{l}, l}\,p_{l}$ for each $l\ge 1$, and that the series $$\sum_{l\ge 1}\left(\dfrac{q_{l}\, p_{l}}{p_{l+1}}\right)^{2}$$ is convergent. Let $(n_{k})_{k\ge 1}$ be the strictly increasing sequence defined by 
\[
\{n_{k}\}=\bigcup_{l\ge 1}\{p_{l},q_{1, l}\, p_{l},\dots,q_{r_{l}, l}\, p_{l}\}
\]
There exists a continuous generalized Riesz product $\sigma $ on $\T$ which is \mbox{IP-Dirichlet} with respect to the \seq\ $(n_{k})_{k\ge 1 }$. 
\end{proposition}
\begin{proof}
 As in the proof of Corollary \ref{Cor4}, we can suppose that $\sum_{k\ge 1}(q_{l}p_{l}/p_{l+1})^{2}<1/324$, and consider a sequence $(\varepsilon _{l})_{l\ge 1}$ going to zero as $l$ tends to infinity with $\varepsilon _{1}=0$ and $0<\varepsilon _{l}<1/2$ for each $l\ge 2$, such that
\[
\sum_{l\ge 1}\Bigl( \dfrac{1}{\varepsilon _{l+1}} \dfrac{q_{l}p_{l}}{p_{l+1}}\Bigr)^{2}<\frac{1}{81}\cdot
\]
The same argument as in the proof of Corollary \ref{Cor4} shows that 
  for $l\ge 1$ the integers $m_{l}=[ ({\varepsilon _{l+1}p_{l+1}-\varepsilon _{l}p_{l}})/({2p_{l}})]$ are positive, and that assumptions (\ref{Eq1}) and (\ref{Eq2}) of Proposition \ref{Prop3} are satisfied. As $m_{l}$ is equivalent as $l$ tends to infinity to $({\varepsilon _{l+1}}{p_{l+1}})/(2{p_{l}})$,
${q_{l}}/({m_{l}+2})$ is equivalent to $  (2q_{l}p_{l})/(\varepsilon _{l+1}{p_{l+1}})$. 
Our assumption implies then that the series 
\begin{equation}\label{eq8bis}
 \sum_{l\ge 1}\Bigl(\frac{q_{l}}{m_{l}+2}\Bigr)^{2}
\end{equation}
 is convergent, and that ${q_{l}\pi }\le{m_{l}+2}$ for each $l\ge 1$. Applying Proposition \ref{Prop3} to the sequence $(p_{l})_{l\ge 1}$, we get a continuous generalized Riesz product $\sigma  $, and the estimates (\ref{Eq3}) yield that
\[
\wh{\sigma }\Bigl( \sum_{l\in F}\Bigl( \sum_{j\in G_{l}}q_{j,\, l}\Bigr)p_{l}\Bigr)\ge\prod_{l\in F}\Bigl( 1-3\pi^{2}\Bigl( \dfrac{q_{l}}{m_{l}+2}\Bigr)^{2}\Bigr)
\]
for each set $F\in \mathcal{F}$ and each subsets $G_{l}$ of $\{0,\dots,r_{l}\}$, $l\in F$. In order to show that the measure $\sigma $ is \mbox{IP-Dirichlet} with respect to $(n_{k})_{k\ge 1}$, it remains to observe that the product on the right-hand side is convergent
by (\ref{eq8bis}). We then conclude as in the proof of Corollary \ref{Cor4}. 
\end{proof}

The proof of Theorem \ref{Cor6} is now a straightforward corollary of Proposition \ref{Prop5}. Recall that we wish to prove that if $(n_{k})_{k\ge 1}$ is a sequence of integers for which there exists an infinite subset $S$ of $\N$ such that
\[
\sum_{k\in S}\, \, \Bigl( \dfrac{n_{k}}{n_{k+1}}\Bigr)^{2}<+\infty \quad\textrm{and}\ n_{k}\vert n_{k+1}\ \textrm{for each}\ k\not\in S,
\]
then there exists a continuous generalized Riesz product $\sigma $ on $\T$ which is \ipd\ with respect to $(n_{k})_{k\ge 1}$.

\begin{proof}[Proof of Theorem \ref{Cor6}]
 Let $\Phi :\N\rightarrow\N$ be a strictly increasing function such that $S=\{\Phi (l),\ l\ge 1\}$. Set $p_{l}=n_{\Phi (l)+1}$ for $l\ge 1$ and write for each $k\in \{\Phi (l)+1,\dots,\Phi (l+1)\}$
\[
n_{k}=s_{0,l}\, s_{1,l}\, \dots s_{k-(\Phi (l)+1),l}\, p_{l},
\]
with $s_{0,l}=1$ and $s_{j,l}\ge 2$ for each $j=1,\ldots, \Phi(l+1)-(\Phi(l)+1)$.
With the notation of Proposition \ref{Prop5} we have $r_{l}=\Phi (l+1)-(\Phi (l)+1)$ and
\[
q_{k-(\Phi (l)+1),l}=s_{0, l}\, s_{1,l}\, \dots s_{k-(\Phi (l)+1), l}
\]
Hence $q_{l}=q_{0,l}+\cdots+q_{r_{l}, l}=s_{0, l}+s_{0,l}\, s_{1,l}+\cdots+s_{0,l}\, s_{1, l}\dots s_{r_{l},l}$. We have
\begin{align*}
 \dfrac{q_{l}}{s_{0, l}\, s_{1, l}\dots s_{r_{l},l}}&=1+\dfrac{1}{s_{r_{l},l}}+\dfrac{1}{s_{r_{l}-1,l}\, s_{r_{l},l}}+\dots+\dfrac{1}{ s_{2,l}\dots s_{r_{l},l}}+\dfrac{1}{s_{1,l}\dots s_{r_{l},l}}\\
&\le 1+\dfrac{1}{2}+\dfrac{1}{4}+\cdots+\dfrac{1}{2^{r_{l}}}\quad\textrm{since}\ s_{j,l}\ge 2\ \textrm{for each}\ j=1,\ldots, r_{l}\\
&\le 2.
\end{align*}
This yields that $q_{l}\le 2 s_{0, l}s_{1,l}\dots s_{r_{l}, l}=2q_{r_{l},l}$ for each $l\ge 1$. Our assumption that the series $\sum_{k\in S}(n_{k}/n_{k+1})^{2}$ is convergent means that the series $\sum_{l\ge 1}(q_{r_{l}, l}\, p_{l}/p_{l+1})^{2}$ is convergent. Hence the series $\sum_{l\ge 1}(q_{l}\, p_{l}/p_{l+1})^{2}$ is convergent and the conclusion follows from Proposition \ref{Prop5}.
\end{proof}

Our next result shows the optimality of the assumption of Proposition \ref{Prop5} that the series $\sum_{l\ge 1}(q_{l}p_{l}/p_{l+1})^{2}$ is convergent.

\begin{proposition}\label{Prop7}
 Let $(\gamma _{l})_{l\ge 1}$ be any sequence of positive real numbers, going to zero as $l$ goes to infinity, such that the series $\sum_{l\ge 1}\gamma _{l}^{2}$ is divergent, with $0<\gamma _{l}<1$ for each $l\ge 2$. Let $(r_{l})_{l\ge 1}$ be a sequence of integers growing to infinity so slowly that the series $\sum_{l\ge 1}\gamma _{l}^{2}/r_{l}$ is divergent, with $r_{l}\ge 2$ for each $l\ge 1$. Define a sequence $(p_{l})_{l\ge 1}$ of integers by setting $p_{1}=1$ and $p_{l+1}=[r^{2}_{l}/\gamma _{l}]\, p_{l}+1$. For each $l\ge 1$, we have $p_{l+1}>r_{l}\, p_{l}$. Define a strictly increasing sequence $(n_{k})_{k\ge 1}$ of integers by setting 
\[
\{n_{k}\}=\bigcup_{l\ge 1}\{p_{l},2\, p_{l},\dots,r_{l}\, p_{l}\}.
\]
Then no continuous measure $\sigma $ on the unit circle can be \mbox{IP-Dirichlet} with respect to the sequence $(n_{k})_{k\ge 1}$. 
\end{proposition}

 \begin{proof}
  We are going to show that $G_{2}((n_{k}))=\{1\}$. It will then follow from Theorem \ref{th0} that no continuous probability measure on $\T$ can be \mbox{IP-Dirichlet} with respect to the sequence $(n_{k})_{k\ge 1}$.  
Suppose that $\lambda \in\T\setminus\{1\}$ is such that 
\begin{equation}\label{ee3}
\sum_{k\ge 1}|\lambda ^{n_{k}}-1|^{2}=\sum_{l\ge 1}\sum_{j=1}^{r_{l}}|\lambda ^{jp_{l}}-1|^{2}<+\infty.
\end{equation}
Let $C$ be a positive constant such that for each $\theta \in[0,1]$, $\frac{1}{C}\{\theta \}\ge|e^{2i\pi \theta }-1|\ge C\{\theta \}$. Writing $\lambda $ as $\lambda =e^{2i\pi \theta }$, $\theta \in[0,1)$, we have that
\begin{equation}\label{ee9bis}
|\lambda ^{jp_{l}}-1|\ge C\{jp_{l}\theta \}\quad\textrm{for each}\ l\ge 1\ \textrm{and}\ j=1,\dots,r_{l}.
\end{equation}
Now $\{\theta p_{l}\}<1/r_{l}$ for sufficiently large $l$. Else the set $\{\{j\theta p_{l}\}, j=1,\dots,r_{l}\}$ would form a $\{\theta p_{l}\}$-dense net of $[0,1]$, and this would contradict the fact, implied by (\ref{ee3}) and (\ref{ee9bis}), that the quantity $\sum_{j=1}^{r_{l}}\{j\theta p_{l}\}^{2}$ tends to zero as $l$ tends to infinity. Hence, for sufficiently large $l$, $\{j\theta p_{l}\}=j\{\theta p_{l}\}$ for every $j=1,\dots,r_{l}$, and thus the series $\sum_{l\ge 1}\sum_{j=1}^{r_{l}}j^{2}|\lambda ^{p_{l}}-1|^{2}$ is convergent. As $r_{l}$ tends to infinity with $l$, this means that the series 
\begin{equation}\label{ee4}
 \sum_{l\ge 1}r_{l}^{3}\, |\lambda ^{p_{l}}-1|^{2}
\end{equation}
is convergent. 
\par
Let now $(\delta _{l})_{l\ge 1}$ be a sequence of real numbers going to zero so slowly that the series $\sum_{l\ge 1}\frac{1}{r_{l}}\, \gamma _{l}^{2}\, \delta _{l}^{2}$ is divergent. Suppose that $|\lambda ^{p_{l}}-1|<\frac{\gamma _{l}}{r_{l}^{2}}\, \delta _{l}$ for infinitely many $l$. Then,
\[
\bigl|\lambda ^{\bigl[ \frac{r_{l}^{2}}{\gamma _{l}}\bigr]\, p_{l}}-1 \bigr|<\delta _{l}\quad\textrm{for all these}\ l,
\]
and by definition of $p_{l+1}$, $|\lambda ^{p_{l+1}}-\lambda |<\delta _{l}$. Letting $l$ tend to infinity along this set of integers, and remembering that $|\lambda ^{p_{l+1}}-1|\to 0$ as $l$ tends to infinity, we get that $\lambda =1$, which is contrary to our assumption. Hence $|\lambda ^{p_{l}}-1|\ge \frac{\gamma _{l}}{r_{l}^{2}}\, \delta _{l}$ for all integers $l$ sufficiently large. Combining this with (\ref{ee4}), this implies that the series $$\sum_{l\ge 1}r_{l}^{3}\, \frac{\gamma _{l}^{2}}{r_{l}^{4}}\, \delta _{l}^{2}=\sum_{l\ge 1}\frac{1}{r_{l}}\,\gamma _{l}^{2}\, \delta _{l}^{2}$$ is convergent, which is again a contradiction. So $G_{2}((n_{k}))=\{1\}$ and we are done.
 \end{proof}

Consider the sets $\{n_{k}\}$ given by Proposition \ref{Prop7}. With the notation of Proposition \ref{Prop5}, $q_{l}$ is equivalent to $r^{2}_{l}/2$ as $k$ tends to infinity, and the series 
$\sum_{l\ge 1}({q_{l}p_{l}}/{p_{l+1}})^{2}$ is divergent because
$({q_{l}p_{l}}/{p_{l+1}})^{2}$ is equivalent to $\gamma _{l}^{2}/4 $. This shows the optimality of the condition given in Proposition \ref{Prop5}.
\par\smallskip
Looking at the construction of Proposition \ref{Prop7} from a different angle yields an example of a sequence $(n_{k})_{k\ge 1}$ such that $G_{2}((n_{k}))$ is uncountable, but still no continuous probability measure on $\T$ can be \mbox{IP-Dirichlet} with respect to $(n_{k})_{k\ge 1}$. This is Theorem \ref{th1}. 

\section{Proof of Theorem \ref{th1}}

Recall that we aim to construct a strictly increasing sequence $(n_{k})_{k\ge 1}$ of integers such that $G_{2}((n_{k}))$ is uncountable, but no continuous \proba\ \mea\ on $\T$ is \ipd\ \wrt\ the \seq\ $(n_{k})_{k\ge 1}$. This \seq\ $(n_{k})_{k\ge 1}$ will be of the kind considered in the previous section.
Consider first the \seq\ $(p_{l})_{l\ge 1}$ defined by 
$p_{1}=1$ and $p_{l+1}=\frac{l^{2}(l^{2}+1)}{2}p_{l}$ for all  $l\ge 1$. We then define the sequence $(n_{k})_{k\ge 1}$ by setting 
$$\{n_{k}\textrm{ ; }k\ge 1\}=\bigcup_{l\ge 2}\{p_{l},2p_{l},\ldots, l^{2}p_{l}\}.$$
As $l^{2}p_{l}<p_{l+1}$ for all $l\ge 2$, the sets $\{p_{l},2p_{l},\ldots, l^{2}p_{l}\}$ are consecutive sets of integers. Let $(M_{l})_{l\ge 1}$ be the unique sequence of integers such that $\{n_{M_{l-1}+1},\ldots, n_{M_{l}}\}=\{p_{l},2p_{l},\ldots, l^{2}p_{l}\}$ for each $l\ge 2$.
We now know (see for instance \cite{BaGr} or \cite{EG} for a proof) that there exists a perfect uncountable subset $K$ of $\T$ such that
$$|\lambda^{p_{l}}-1|\le C\dfrac{p_{l}}{p_{l+1}}\quad \textrm{ for all } \lambda\in K \textrm{ and } l\ge 2,$$
where $C$ is a positive universal constant. Hence for $\lambda\in K$, $l\ge 2$ and $j\in \{1,\ldots, l^{2}\}$ we have
$$|\lambda^{jp_{l}}-1|\le C\, j\,\dfrac{p_{l}}{p_{l+1}}\le 2C\, l^{2}\,\dfrac{1}{l^{4}}=\dfrac{2C}{l^{2}}\cdot$$
Thus 
$$\sum_{j=1}^{l^{2}}|\lambda^{jp_{l}}-1|^{2}\le l^{2}\,\dfrac{4C^{2}}{l^{4}}=\dfrac{4C^{2}}{l^{2}}\cdot$$
Hence the series $\sum_{l\ge 2}\sum_{j=1}^{l^{2}}|\lambda^{jp_{l}}-1|^{2}$ is convergent for all $\lambda \in K$, that is the series
$\sum_{k\ge 1}|\lambda^{n_{k}}-1|^{2}$ is convergent for all $\lambda \in K$. We have thus proved the first part of our statement, namely that $G_{2}((n_{k}))$ is uncountable.

\par\smallskip

Let now $\sigma$ be a continuous \proba\ \mea\ on $\T$. The proof that $\sigma$ cannot be \ipd\ \wrt\ the \seq\ $(n_{k})_{k\ge 1}$ relies on the following lemma:

\begin{lemma}\label{lem1}
For all $l\ge 2$ and all $s\ge 1$, $sp_{l}$ belongs to the set $$\{\sum_{k\in F}n_{k}\textrm{ ; } F\in \mathcal{F},\;
\min(F)\ge M_{l-1}+1\}.$$ 
\end{lemma}

\begin{proof}[Proof of Lemma \ref{lem1}]
It is clear that for all $n\ge 1$,
$$\{\sum_{j\in F} j\textrm{ ; } F\subseteq\{1,\ldots,n\},\, F\not =\varnothing\}=\{1,\ldots, \frac{n(n+1)}{2}\}.$$ Hence
$$\{\sum_{j\in F} jp_{l}\textrm{ ; } F\subseteq\{1,\ldots,l^{2}\},\, F\not =\varnothing\}=\{p_{l},2p_{l}\ldots, \frac{l^{2}(l^{2}+1)}{2}p_{l}\}, $$ i.e.
$$\{\sum_{k\in F} n_{k}\textrm{ ; } F\subseteq\{M_{l-1}+1,\ldots,M_{l}\},\, F\not =\varnothing\}=\{p_{l},2p_{l}\ldots,p_{l+1}\}.$$
This proves Lemma \ref{lem1} for $s\in \{1,\ldots, \frac{l^{2}(l^{2}+1)}{2}\}$. Then since
$$\{\sum_{k\in F} n_{k}\textrm{ ; } F\subseteq\{M_{l}+1,\ldots,M_{l+1}\},\, F\not =\varnothing\}=\{p_{l+1},2p_{l+1}\ldots,
\frac{(l+1)^{2}((l+1)^{2}+1)}{2}p_{l+1}\},$$ we get that
\begin{eqnarray*}
 \{\sum_{k\in F} n_{k}&\textrm{;}& F\subseteq\{M_{l-1}+1,\ldots,M_{l+1}\},\, F\not =\varnothing\}\\
&=&\{p_{l}, 2p_{l}, \ldots, p_{l+1},p_{l+1}+p_{l},p_{l+1}+2p_{l}, \ldots \\
& &\hspace*{1.9cm} 2p_{l+1},\ldots,
\frac{(l+1)^{2}((l+1)^{2}+1)}{2}p_{l+1}\}
\\
&=&\{p_{l}, 2p_{l}, \ldots,\frac{l^{2}(l^{2}+1)}{2}\,\cdot\,\frac{(l+1)^{2}((l+1)^{2}+1)}{2}p_{l}\}.
\end{eqnarray*}
Continuing in this fashion we obtain that for all $q\ge 1$,
\begin{eqnarray*}
 \{\sum_{k\in F} n_{k}&\textrm{;}& F\subseteq\{M_{l-1}+1,\ldots,M_{l+q}\},\, F\not =\varnothing\}\\
&=&\{p_{l}, 2p_{l}, \ldots,\prod_{j=0}^{q}\frac{(l+j)^{2}((l+j)^{2}+1)}{2}p_{l}\}.
\end{eqnarray*}
The conclusion of Lemma \ref{lem1} follows from this.
\end{proof}

Suppose now that $\sigma$ is \ipd\ \wrt\ $(n_{k})_{k\ge 1}$. Let $l_{0}\ge 2$ be such that for every $F\in \mathcal{F}$ with $\min(F)\ge M_{l_{0}-1}+1$,
$|\hat{\sigma}(\sum_{k\in F}n_{k})|\ge {1}/{2}$.
Then Lemma \ref{lem1} implies that for all $s\ge 1$, $|\hat{\sigma}(sp_{l_{0}})|\ge {1}/{2}$. This contradicts the continuity of the \mea\ $\sigma$.

\section{Additional results and comments}

\subsection{A remark about the Erd\"os-Taylor sequence}
Let $(n_{k})_{k\ge 1}$ be the \seq\ of integers defined by $n_{1}=1$ and $n_{k+1}=kn_{k}+1$ for every $k\ge 1$. This \seq\ is interesting in our context because $G_{1}((n_{k}))=\{1\}$ while $G_{2}((n_{k}))$ is uncountable 
(\cite{ET}, see also \cite{AHL}): if $\lambda\in\T\setminus\{1\}$, there exists a positive constant $\varepsilon$ such that $|\lambda^{n_{k}}-1|\ge\frac{\varepsilon}{k}$ for all $k\ge 1$. Indeed, if for some $k$ we have $|\lambda^{n_{k}}-1|\le\frac{\varepsilon}{k}$, then $|\lambda^{kn_{k}}-1|\le {\varepsilon}$, so that
$|\lambda^{n_{k+1}}-1|\ge |\lambda-1|-{\varepsilon}$. This is a contradiction for $\varepsilon=\frac{1}{2}|\lambda-1|$. Hence if $\lambda\in\T\setminus\{1\}$ the series $\sum_{k\ge 1}|\lambda^{n_{k}}-1|$ is divergent. On the other hand, since the series $\sum_{k\ge 1}({n_{k}}/{n_{k+1}})^{2}$ is convergent, $G_{2}((n_{k}))$ is uncountable. It is proved in \cite{AHL} that there exists a continuous \proba\ \mea\ $\sigma$ on $\T$ which is \ipd\ \wrt\ $(n_{k})_{k\ge 1}$. This statement can also be seen as a consequence of Theorem 2.2 of \cite{K}: it is shown there that there exists a continuous generalized Riesz product $\sigma$ on $\T$ and a $\delta >0$ such that 
$$|\hat{\sigma}(\sum_{k\in F}n_{k})|\ge \delta$$ for every $F\in \mathcal{F}$ such that $\min(F)>4$. It is not difficult to see that this \mea\ $\sigma$ is in fact\ \ipd\ \wrt\ $(n_{k})_{k\ge 1}$. We briefly give the argument below. It can be generalized to all sequences $(n_{k})_{k\ge 1}$ such that the series $\sum_{k\ge 1}(n_{k}/n_{k+1})^{2}$ is convergent, thus yielding another proof of Corollary \ref{Cor4}.
\par\smallskip
The \mea\ $\sigma$ of \cite{K} is constructed in the following way: let $\Delta$ be the function defined for $t\in\R$ by $\Delta(t)=\max(1-6|t|,0)$. If $K$ is the function  $\R$ given by the expression
$$K(t)=\dfrac{1}{2\pi}\left(\dfrac{\sin \frac{t}{2}}{\frac{t}{2}}\right)^{2}, \quad t\in\R$$ and $K_{\alpha}$ is defined for each $\alpha>0$ by $K_{\alpha}(t)=\alpha K(\alpha t)$, $t\in\R$, then $\Delta(x)=\hat{K}_{\frac{1}{6}}(x)$ for every $x\in\R$.
The function $\Delta\ast\Delta$ is a $\mathcal{C}^{2}$ function on $\R$ which is supported on $[-\frac{1}{3},\frac{1}{3}]$, takes positive values on $]-\frac{1}{3},\frac{1}{3}[$, and attains its maximum at the point $0$. Hence its derivative vanishes at the point $0$. Let $a>0$ be such that the function $\varphi=a\Delta\ast\Delta$ satisfies $\varphi(0)=1$. We have also $\varphi'(0)=0$, and so there exists a constant $c\ge 0$ and a $\gamma\in (0,\frac{1}{3})$ such that for all $x$ with $|x|<\gamma$, $\varphi(x)\ge 1-cx^{2}$. Lastly, recall that $\varphi(x)=a\widehat{K^{2}_{\frac{1}{6}}}(x)$ for all $x\in\R$. Consider now the \seq\ $(P_{j})_{j\ge 1}$ of trigonometric polynomials defined on $\T$ in the following way: for $j\ge 1$ and $t\in\R$,
$$P_{j}(e^{it})=\sum_{s\in\Z}\varphi(\dfrac{s}{j})e^{ist}.$$ This is indeed a polynomial of degree at most $\lfloor\frac{j}{3}\rfloor$, since $\varphi(\frac{s}{j})=0$ as soon  as $\frac{s}{j}\ge \frac{1}{3}$. We now claim that $P_{j}$ takes only nonnegative values on $\T$: indeed, consider for each $j\ge 1$ and $t\in \R$ the function $\Phi_{j,t}$ defined by
$\Phi_{j,t}(x)=jK^{2}_{\frac{1}{6}}(j(x+t))$, $x\in\R$. Its Fourier transform is then given by
$\hat{\Phi}_{j,t}(\xi)=e^{i\xi t}\widehat{K^{2}_{\frac{1}{6}}}(\frac{\xi}{j})=e^{i\xi t}\Delta\ast\Delta(\frac{\xi}{j})$. Thus $P_{j}(e^{it})=a\sum_{s\in\Z}\hat{\Phi}_{j,t}(s)$. Applying the Poisson formula to the function $\Phi_{j,t}$, we get that
$P_{j}(e^{it})=2\pi a\sum_{s\in\Z}\Phi_{j,t}(2\pi s)=2\pi a\sum_{s\in\Z} j {K}_{\frac{1}{6}}^{2}(j(2\pi s+t))\ge 0.$ Hence $P_{j}(e^{it})$ is nonnegative for all $t\in\R$, $\hat{P}_{j}(0)=1$ and $\hat{P}_{j}(1)=\varphi(\frac{1}{j})\ge 1-\frac{c}{j^{2}}$ as soon as $j\ge j_{0}$, where $j_{0}=\lfloor\frac{1}{\gamma}\rfloor+1$. Consider then for $m\ge j_{0}$ the nonnegative polynomials $Q_{m}$ defined by 
$$Q_{m}(e^{it})=\prod_{j=j_{0}}^{m}P_{j}(e^{in_{j}t}), \quad t\in\R.$$
We have $\hat{Q}_{m}(0)=1$ for each $m\ge 1$. Moreover, since the degree of $P_{j}$ is less than $\lfloor \frac{j}{3}\rfloor$ and $n_{j+1}>\frac{jn_{j}}{3}$, the polynomials $Q_{m}$ converge in the $w^{*}$-topology to a generalized Riesz product $\sigma$ on $\T$ which is continuous and such that for every set  $F\in \mathcal{F}$ with $\min(F)\ge j_{0}$,
$$\hat{\sigma}(\sum_{k\in F}n_{k})\ge\prod_{k\in F}(1-\frac{c}{k^{2}})\cdot$$
It follows that $\sigma$ is an \ipd\ \mea\ \wrt\ the \seq\ $(n_{k})_{k\ge 1}$.

\subsection{A  sequence $(n_{k})_{k\ge 1}$ \wrt\ which there exists a continuous Dirichlet measure, but such that $G_{\infty}((n_{k}))=\{1\}$}
The examples of \seq s $(n_{k})_{k\ge 1}$ given in \cite{BDLR} and \cite{EG} for which there exists a continuous \proba\ \mea\ $\sigma  $ on $\T$ such that $\hat{\sigma  }(n_{k})\to 1$
as $k\to +\infty$ all share the property that $|\lambda ^{n_{k}}-1|\to 0$ for some $\lambda \in\T\setminus\{1\}$. One may thus wonder whether there exists a sequence $(n_{k})_{k\ge 1}$ \wrt\  which there exists a continuous Dirichlet \proba\ \mea\ $\sigma  $, and such that 
$G_{\infty}((n_{k}))=\{\lambda \in \T \textrm{ ; } |\lambda ^{n_{k}}-1|\to 0\}=\{1\}$. The answer is yes, and an ad hoc \seq\ $(n_{k})_{k\ge 1}$ can be constructed from the Erd\"os-Taylor \seq\ above. Changing notations, let us denote by $(p_{k})_{k\ge 1}$ this \seq\, defined by $p_{1}=1$ and $p_{k+1}=kp_{k}+1$ for each $k\ge 1$. For each integer $q\ge 1$, consider the finite set 
$$\mathcal{P}_{q}=\{\sum_{k\in F}p_{k} \textrm{ ; }F\not =\emptyset, \; F\subseteq\{2^{q}+1,\ldots, 2^{q+1}\}\}.$$
The set $\bigcup_{q\ge 1} \mathcal{P}_{q}$ can be written as $\{n_{k} \textrm{ ; }k\ge 1\}$, where $(n_{k})_{k\ge 1}$ is a strictly increasing \seq\ of integers. Let now $\sigma  $ be a continuous \proba\ \mea\ which is \ipd\ \wrt\ the Erd\"os-Taylor \seq\ $(p_{k})_{k\ge 1}$:
$$\hat{\sigma  }(\sum_{k\in F}p_{k})\to 1 \quad\textrm{ as }\quad \min(F)\to+\infty, \; F\in \mathcal{F}.$$ This implies that $\hat{\sigma  }(n_{k})\to 1$ as $k\to +\infty$. Indeed, let $\varepsilon >0$ and $k_{0}$ be such that $|\hat{\sigma  }(\sum_{k\in F}p_{k})-1|<\varepsilon $ for all $F\in \mathcal{F}$  with $\min(F)\ge k_{0}$. Let $q_{0}$ be such that $2^{q_{0}}+1\ge k_{0}$. Then $|\hat{\sigma  }(n_{k})-1|<\varepsilon $ for all $k$ such that $n_{k}$ belongs to the union $\bigcup_{q\ge q_{0}}\mathcal{P}_{q}$. Since all the sets $\mathcal{P}_{q}$ are finite, $|\hat{\sigma  }(n_{k})-1|<\varepsilon $ for all but finitely many $k$.

It remains to prove that $G_{\infty}((n_{k}))=\{1\}$, and the argument for this is very close to one employed in \cite{AHL}. Let $\varepsilon \in (0,1/16)$ for instance, and suppose that $\lambda \in\T$ is such that $|\lambda ^{n_{k}}-1|<\varepsilon $ for all $k$ larger than some $k_{0}$. We claim then that if $q_{0}$ is such that $2^{q_{0}}+1\ge k_{0}$, then we have for all $q$ larger than $q_{0}$
\begin{equation}\label{ee10}
\sum_{k=2^{q}+1}^{2^{q+1}}|\lambda ^{p_{k}}-1|<2C^{2}\varepsilon ,
\end{equation}
where $C>0$ is a constant such that $\{t\}/C\le |e^{2i\pi t}-1|\le C\{t\}$ for all $t\in\R$.
Indeed, our assumption that $|\lambda ^{n_{k}}-1|<\varepsilon $ for all $k\ge k_{0}$ implies that for all $q\ge q_{0}$ and all disjoint finite subsets $F$ and $G$ of the set $\mathcal{P}_{q}$,
$$\{\sum_{k\in F}p_{k}\theta\}<C\varepsilon , \quad \{\sum_{k\in G}p_{k}\theta \}<C\varepsilon \quad \textrm{and}\quad \{\sum_{k\in F\sqcup G}p_{k}\theta \}<C\varepsilon $$
where $\lambda =e^{2i\pi\theta }$ with $\theta \in [0,1)$ and $F\sqcup G$ denotes the disjoint union of $F$ and $G$. Now the same argument as in \cite[Prop. 1.1]{AHL} yields that
$$\langle\sum_{k\in F\sqcup G}p_{k}\theta\rangle= \langle\sum_{k\in F}p_{k}\theta \rangle +\langle\sum_{k\in G}p_{k}\theta \rangle.$$ Setting
$$A_{q,+}=\{k\in\{2^{q}+1,\ldots 2^{q+1}\} \textrm{ ; }\langle p_{k}\theta \rangle \ge 0\}$$ and $$A_{q,-}=\{k\in\{2^{q}+1,\ldots 2^{q+1}\} \textrm{ ; }\langle p_{k}\theta \rangle < 0\},$$ this implies that 
$$\sum_{k\in A_{q,+}}\{p_{k}\theta \}<C\varepsilon \quad \textrm{and}\quad \sum_{k\in A_{q,-}}\{p_{k}\theta \}<C\varepsilon.$$
Hence
$$
\sum_{k=2^{q}+1}^{2^{q+1}}\{p_{k}\theta \}<2C\varepsilon\quad \textrm{so that}\quad
 \sum_{k=2^{q}+1}^{2^{q+1}}|\lambda ^{p_{k}}-1|<2C^{2}\varepsilon\quad \textrm{for all }q\ge q_{0}.
$$
Suppose now that $\lambda \not =1$, and set $\varepsilon =|\lambda -1|/(4C^{2})$. Then (\ref{ee10}) above implies that there exists an infinite subset 
$E$ of $\N$ such that $|\lambda ^{p_{k}}-1|\le (2C^{2}\varepsilon )/k$ for all $k\in E$. If it were not the case, we would have $|\lambda ^{p_{k}}-1|>(2C^{2}\varepsilon )/k$ for all $k$ large enough, so that
\begin{equation}\label{ee11}
 \sum_{k=2^{q}+1}^{2^{q+1}}|\lambda ^{p_{k}}-1|>2C^{2}\varepsilon\, \sum_{k=2^{q}+1}^{2^{q+1}}\frac{1}{k}\ge 2C^{2}\varepsilon \,\frac{2^{q+1}-2^{q}}{2^{q}}\ge 2C^{2}\varepsilon
\end{equation}
 for all $q$ large enough, which is a contradiction with (\ref{ee10}). This proves the existence of the set $E$. Now for all $k\in E$
$$|\lambda ^{p_{k+1}}-1|\ge |\lambda -1|-|\lambda ^{kp_{k}}-1|\ge |\lambda -1|-k|\lambda ^{p_{k}}-1|\ge 4C^{2}\varepsilon -2C^{2}\varepsilon=2C^{2}\varepsilon  .$$ But this stands again in contradiction with (\ref{ee10}), and we infer from this that $\lambda $ is necessarily equal to $1$. Thus $G_{\infty}((n_{k}))=\{1\}$, and we are done.

\end{document}